\documentclass[11pt,twoside]{amsart}
\usepackage{latexsym,amsfonts,amsmath,amssymb,soul}

\numberwithin{equation}{section}

\makeatletter
\@namedef{subjclassname@2020}{%
  \textup{2020} Mathematics Subject Classification}
\makeatother

\catcode`\@=11

\renewcommand{\section}%
   {\setcounter{equation}{0}\@startsection {section}{1}{\z@}{-3.5ex plus -1ex
  minus -.2ex}{2.3ex plus .2ex}{\Large\bf}}

\newtheorem{Th}{Theorem}[section]
\newtheorem{Rem}[Th]{Remark}

\newtheorem{Lemma}[Th]{Lemma}
\newtheorem{Def}[Th]{Definition}
\newtheorem{Prop}[Th]{Proposition}
\newtheorem{Cor}[Th]{Corollary}

\usepackage{calrsfs}
\DeclareMathAlphabet{\pazocal}{OMS}{zplm}{m}{n}

\newcommand{\Pcal}{\pazocal{P}}

\def\R{\mathbb R}
\def\N{\mathbb N}
\def\sp{\hskip 0.5 pt}
\def\spp{\hskip 1pt}

\def\spi{\hskip 2pt}

\title{A Kirchhoff equation with infinite conservation laws}

\author[Boiti]{Chiara Boiti}
\address{
Dipartimento di Matematica e Informatica \\Universit\`a di Ferrara\\
Via Ma\-chia\-vel\-li n.~30\\
I-44121 Ferrara\\
Italy}
\email{chiara.boiti@unife.it}

\author[Manfrin]{Renato Manfrin}
\address{
Dipartimento di Culture del Progetto, Universit\`a IUAV di Venezia\\
Dorsoduro n.~2196, I-30123 Venezia, 
 Italy}

\email{manfrin@iuav.it}

\begin{document}

\subjclass[2020]{Primary: 35L65, 35L72; Secondary: 35L15}
\keywords{Kirchhoff equation, conservation laws, quasilinear hyperbolic equations.}

\begin{abstract}
We show here that the quasilinear Kirchhoff-Pokhozhaev equation 
$$u_{tt}-   \big ( a\int_{\R^n} |\nabla u |^2 dx + b \big )^{-2 }  \Delta u = 0\sp , 
$$ 
with $\sp  a\ne 0 \sp $, admits conservation laws of all orders. 
 \end{abstract}

\maketitle

\markboth{\sc  A Kirchhoff equation with infinite conservation laws}
 {\sc C.~Boiti, R.~Manfrin}


\section{Introduction and results} 

 We consider here the Kirchhoff\spp -\spp Pokhozhaev equation 
 \begin{equation}\label{KP}
  u_{tt}- \frac {\Delta u} {\big ( a\sp \| \nabla u \|^2+ b \big )^2 }  = 0 
  \quad \text{in} \quad  \R^n \times [0, T)  ,
   \end{equation}
 for some $\sp  T> 0\spp $, where $ \spp\|\nabla u \|^2 = \int_{\R^n} |\nabla u (x,t) |^2 \spp dx\spp $ 
  and $\sp a, \sp b \in \R \sp $, with $a\neq0$.

S.\spi I.\,Pokhozhaev proved in \cite{P1}, \cite{P3} that equation \sp \eqref{KP} \sp admits a second-order conservation law.
More precisely, let $\spp u \spp $ be  defined in $\sp \R^n \times [0,T)\sp$, for some $\sp T >0 \sp $, and such that
\begin{equation}\label{regol2}
  u \in C^j([0,T); H^{2-j}(\R^n))  \quad  \quad (j= 0,1) \sp ,
\end{equation}
and
\begin{equation} \label{reg2}
 \phantom{ \int_{\R}} q \,  =  \,   a \spp \|\nabla u \|^2 +b \ne 0 \quad \text{for} \quad t\in [0,T);\phantom{ \int_{\R}}
 \end{equation} 
if $\spp u \spp $ solves  \eqref{KP} in $\spp \R^n \times [0,T)$, then the second-order functional
\begin{equation}\label{CL2}
I_2(u) \, := \, q \sp  \| \nabla u_t\|^2 +
 \frac { \|\Delta u \|^2} 
{ q} 
 -   a \spp \left ( \int_{\R^n} \nabla u \cdot \! \nabla u_t \spp dx \right )^2
\end{equation}
remains constant in $\sp\sp  [0,T)$. See also \cite {BM} and \cite{PE} for other derivations of  \eqref{CL2}. 

 As shown first in \cite{P3} and then in \cite{PN} and \cite{PE}, the  conservation of the functional $\sp I_2 \sp $ 
 gives, under appropriate assumptions, the global solvability of the Cauchy problem (and also of the initial boundary value problem) for  equation \eqref{KP} when the initial data $\sp u(x,0)\sp$ and $ \sp u_t(x,0)\sp $ are in Sobolev spaces of suitable order.
 Note that the simple energy conservation (functional of the first order), valid for any 
  Kirchhoff-type equation \begin{equation} \label{KK}
 \, u_{tt} - m (\|\nabla u \|^2) \Delta u= 0 \quad  \text{with}\quad 
 m \in C^0 , \;\,  m > 0\sp ,
 \end{equation}
  does not seem to be sufficient to ensure global solvability. See \cite{AP}.

Recently, in \sp \cite{BM2}, \sp it was shown that \eqref{KP}  also admits a third-order conservation law if one assumes some additional regularity on the solution. Namely, if 
$\spp u \in C^j([0,T); H^{3-j}(\R^n))$, with $j\in\{0,1\}$, 
then the third-order functional
 \begin{equation}\label{CL3}
\begin{aligned}
I_3(u)\, &:=  \, q {\|\Delta u_t\|^2} + \frac{\|\nabla (\Delta u)\|^2} q - 
{q'} \!\int_{\R^n} \Delta u \sp\Delta u_t\spp dx 
 \\
 &\quad\; + \frac{1} 8\spp  {q'}^2 \spp 
  \left ( { q\| \nabla u_t\|^2}  +  \frac{\| \Delta u\|^2} {q}\right )
  -\frac{ 1 }{16\sp a}  \left (  \frac { {q'}^4} 4+ {q^2\spi {q''}^2} \right ) 
 \end{aligned}
 \end{equation}
 remains constant in $[0,T)$. Note that \eqref{CL3} depends on $\spp q, q'\spp$ and also  
 $\spp q''\sp $, which is given by
 \begin{equation*}
 q'' =  2a \left (\| \nabla u_t\|^2 - \frac {\| \Delta u \|^2} {q^2} \right ),
 \end{equation*}
 if  $\sp u \sp$ is a solution of \eqref{KP} that satisfies \eqref{regol2}, \eqref{reg2}.
 
The functional $I_3$ is of the third order, since it is the sum of a principal part
  $$
  \sp q {\|\Delta u_t\|^2} + \frac{\|\nabla (\Delta u)\|^2} q - 
{q'} \!\int_{\R^n} \Delta u \sp\Delta u_t\spp dx 
$$
that depends on third-order derivatives of $u$, and a functional
 $$
J(u) \spp = \spp \frac{1} 8\spp  {q'}^2 \spp 
  \left ( { q\| \nabla u_t\|^2}  +  \frac{\| \Delta u\|^2} {q}\right )
  -\frac{ 1 }{16\sp a}  \left (  \frac { {q'}^4} 4+ {q^2\spi {q''}^2} \right ) 
  $$
which contains only derivatives of $u$ up to the second order.

 \medskip

In this paper, we   show that  the Kirchhoff-Pokhozhaev equation \eqref{KP} has conservation laws  
(more precisely, time-invariant functionals) of any order  $ \sp k\in \N$, $\sp k \ge 3 \sp $,
provided that the considered solutions are sufficiently regular. 

\vskip 0.2cm

To be more precise, given a sufficiently regular solution of \eqref{KP}, which satisfies \eqref{reg2}, we define a functional of order $ \spp \le l$ according to the following:

\begin{Def}\label{lf}
Let $\sp l \in\N\sp$, $\sp l \ge 2\sp $. 
 We say that  $\sp J= J(u)\sp$ is a  functional   of order $\spp  \le l \spp \sp$ if  $\, J\,$ can be expressed in the form
   \begin{equation}
   \label{defju}
  \begin{aligned}
 J\, =  \, J(u)\,=\,P \Big  (\spp q^{-1}\sp , &\,\spp q \spp , \spp q', \spp  \dots ,\sp q^{(2l-2)}, \,\spp
  \|\nabla u \|^2, \dots , \spp \|\nabla^{l} u \|^2,\, \spp  \|\nabla u_t \|^2 , \dots , \|\nabla^{l-1} u_t \|^2  ,
 \\
 &\quad\quad  \int_{\R^n}  \nabla u  \cdot \! \nabla u_t  \spp dx \, \spp , \dots , \,
  \int_{\R^n}  \nabla^{l-1} u  \cdot \! \nabla^{l-1} u_t \spp dx \Big )
  \end{aligned}
  \end{equation}
where $\sp P \sp $ is a polynomial of its arguments and where, for $\sp i \in \N\sp$ and $\sp v \sp $ a given 
function, we have adopted the following convention:
$$
\nabla^{2i} v= \Delta^{ i} v \quad  \quad   \text{and} \quad  \quad 
 \nabla^{2 i+1}v=  \nabla (\Delta^{ i} v )\sp .
$$
\end{Def}



\begin{Rem} \label{lorder} 
To clarify Definition~\ref{lf}, we also note the following facts which can be easily proved.

Let $\sp l\in \N\sp,\sp l \ge 2\sp$. If
$$
  u \in C^j([0,T); H^{l-j}(\R^n))   \quad  \quad (j= 0,1) \sp ,
  $$
  is a solution of \eqref{KP} in $\spp \R^n \times [0,T)$ that
 satisfies condition \eqref{reg2}, then 
 
 \vskip 0.2cm
 \noindent
 $i )\phantom{ii)} q \sp $ is a $\sp  C^{2l-2}\sp $ function in $\sp [0,T)\sp$;
 
  \vskip 0.2cm
 \noindent
$ii)\phantom{i)}q', q'', \dots , q^{(2l-2)}\spp$ can be expressed as polynomials in 
$\spp q^{-1}\spp $, 
$\,
\|\nabla u \|^2$, $ \dots$, $\|\nabla^{l} u \|^2 
\, $,
$\,
\|\nabla u_t \|^2$, $\dots$, $\|\nabla^{l-1} u_t \|^2 
\, $ and
$$
\int_{\R^n}  \nabla u \! \cdot \! \nabla u_t  \, dx \,\,  , \dots , \,\,
  \int_{\R^n}  \nabla^{l-1} u \! \cdot \! \nabla^{l-1} u_t \, dx \spp .
  $$

  This justifies the expression \eqref{defju} of $J(u)$ as a functional of order $\leq l$, because it 
 depends on the derivatives of $\spp \sp u \sp $ of order at most $\spp l $.
\end{Rem}

Now we can state our main result:
\begin{Th}\label{T1}
 Let $k\in \N$, $k\ge 3\sp$. There exists a  functional $\, J_{k}= J_k(u)\sp $, of order $\sp \le k-1\sp $, 
such that, given a function
\begin{equation*}
  u \in C^j([0,T); H^{k-j}(\R^n))  \quad  \quad (j= 0,1) \sp ,
  \end{equation*}
which satisfies \eqref{reg2},  if  $\sp u \sp $  is a solution of \eqref{KP} then   
  \begin{equation}\label{Ik}
I_k (u)\sp =  \sp q {\|\nabla^{k-1} u_t\|^2} + \frac{\|\nabla^{k} u\|^2} q -
 {q'} \!\int_{\R^n} \nabla^{k-1} u  \cdot \!\nabla^{k-1} u_t\spp\sp dx \spp + J_{k}(u)
 \end{equation}
 remains constant in $\sp [0,T)\sp$.
\end{Th}

\vskip 0.2cm

Note that the principal part of $\sp I_k\sp$
$$
\,  q {\|\nabla^{k-1} u_t\|^2} + \frac{\|\nabla^{k} u\|^2} q 
-{q'} \!\int_{\R^n} \nabla^{k-1} u  \cdot \!\nabla^{k-1} u_t\spp dx 
$$ 
depends on  the $\sp k$-order derivatives 
$\sp \nabla^{k} u\sp $ and $\spp \nabla^{k-1}u_t\sp $ while, by Remark~\ref{lorder},  the functional 
$\sp J_{k}\sp $ depends only on lower order derivatives of the solution $\spp u\spp$, that is
$$
 \nabla^j u \; \, \sp \text{for} \; \,  1\le j \le k-1\sp 
 \quad \text{and}
 \quad   \nabla^{j} u_t 
 \; \, \sp \text{for}
 \; \,
  1 \le j \le k-2.
 $$

Finally, we recall that in \cite{P1} S.\spi I.\,Pokhozhaev also proved the remarkable fact that \eqref{KP} is the unique case 
where the Kirchhoff equation \eqref{KK} admits a second order conservation law. In other words, if  \eqref{KK} 
has a second order conservation law, then
\begin{equation}
m(s) = \frac 1 {(as+b)^2}
\end{equation}
for some $\sp a, b \in \R\sp$ not both zero. 
For this reason, it is likely that among the Kirchhoff-type equations \eqref{KK} only \eqref{KP} admits conservation laws of order higher than the second.


\vskip 0.3cm
\section{A quadratic form for a Liouville type equation} \label{LL}


In this section we consider the Liouville-type (cf. \cite[\S\sp 6.2]{MU}) differential equation
 \begin{equation}\label{lin2}
w_{tt} +\frac 1{q^2}\sp |\xi |^2\sp  w=0 \quad \text{for}  \quad  t\in [0,T),
\end{equation}
where $\xi \in \R^n\sp$  and $\, q=q(t): [0,T) \rightarrow \R \, $ is  a sufficiently regular function such that 
\begin{equation}  \label{qno}
q(t) \ne 0 \quad \text{in} \quad [0,T).
\end{equation}
This is obtained from equation \eqref{KP} by partial Fourier transform with respect to the $x$-variable.

Given a complex-valued solution of \eqref{lin2}, say
$$
w= w(\xi, t) \sp,
$$ 
we introduce, as in  \cite[\S\sp 2]{M1}, the following  set of quadratic forms: 
 \begin{Def} \label{Qdef} Given  $\sp k\in\N\sp $, $\spp k \ge 3\sp $,  let
\begin{equation}
 \begin{aligned}\label{Q}
{\pazocal E}_k(\xi, t)\sp  := \spp &\sum_{i=0}^{k-2} \alpha_i \sp |\xi|^{2k-2i-2} \left(|w_t|^2  + \frac 1{q^2}|\xi|^2  |w|^2\right ) 
+\sum_{i=0}^{k-2} \beta_i \sp |\xi|^{2k-2i-2}  \Re  (\overline{ w} \sp w_t)
\\
& \quad \quad + \sum_{i=0}^{k-3} \gamma_i \sp  |\xi|^{2k -2i-4} |w_t|^2,
 \end{aligned}
\end{equation}
where  $ \alpha_i=\alpha_i(t) , \beta_i=\beta_i(t), \gamma_i =\gamma_i(t)$ 
are suitable  $C^1$ real  functions in $[0,T)$. 
\end{Def}

 We will 
 determine $\spi \alpha_0, \beta_0, \gamma_0  \spi$, $\dots $, $\sp \alpha_{k-3}, \beta_{k-3}, \gamma_{k-3}\spi $ and, finally,  $\spi \alpha_{k-2}, \beta_{k-2} \spi$ by equating successively to zero, from highest to lowest, powers of $\sp |\xi|\sp $ in 
the expression of $\spi \frac d{dt}{\pazocal E}_k(\xi, t)$. Our aim is to obtain
\begin{equation} \label{derek0}
{d\over dt} {\pazocal E}_k(\xi, t)=  \beta_{k-2}' |\xi|^2\Re (\overline{w} w_t) \quad \text{for} \quad t\in [0,T).
\end{equation} 

We shall first establish the system of conditions that $\spi \alpha_i, \beta_i, \gamma_i \spi $ must satisfy, and then prove that the system is always solvable.

\hskip 1mm

$\quad 1)$  To begin with, we consider the terms of \eqref{Q} with $i=0$. Since $\spp w \spp $ is a solution of \eqref{lin2}, 
deriving with respect to $\sp t\sp$ we easily find:
\begin{equation} \label{der1}
 \begin{aligned}
\frac d{dt} &\left [\alpha_0 |\xi|^{2k-2} \left (|w_t|^2  + \frac 1{q^2} |\xi|^2  |w|^2\right ) 
+ \beta_0 |\xi|^{2k-2} \Re  (\overline{ w}\sp w_t) + \gamma_0  |\xi|^{2k-4} |w_t|^2 \right ]
\\
&\quad =\,\left [\left ( \frac{\alpha_0}{q^2} \right )'- \frac {\beta_0} {q^2}\right ]|\xi|^{2k} |w|^2 + \left [ \alpha_0' + \beta_0\right ]|\xi|^{2k-2} |w_t|^2  
\\
&\quad \quad \quad \quad  + 
\left [\beta_0' - 2 \frac{\gamma_0}{q^2}\right ] |\xi|^{2k-2} \Re (\overline{w}\sp w_t) +
 \gamma_0' |\xi|^{2k-4} |w_t|^2.
\end{aligned}
\end{equation}
We then require that $\alpha_0, \beta_0, \gamma_0$ satisfy the conditions
\begin{equation} \label{S1}
{\begin{cases}
\displaystyle{\left ( \frac{\alpha_0}{q^2} \right )'- \frac {\beta_0} {q^2}=0}
 \\
 \displaystyle{\alpha_0' + \beta_0 =0 \phantom{\int^8_8} }
 \\
 \displaystyle{\beta_0' - 2 \frac{\gamma_0}{q^2} =0\sp.}
 \end{cases}}
 \end{equation}
In this way the derivative \eqref{der1} reduces to $\,  \gamma_0' |\xi|^{2k-4} |w_t|^2\,$.

\hskip 1mm

$\quad 2)$ If $\sp k \ge 4 \sp $ (so that $k-3 \ge 1$), we consider also the derivatives of the terms  of 
\eqref{Q} with 
index $\sp i \sp $ such that $ 1\le i \le k-3$. We have:
 \begin{equation}  \label{der2}
  \begin{aligned}
\frac d{dt} &\left [\alpha_i |\xi|^{2k -2i-2} \left ( |w_t|^2  + \frac 1{q^2} |\xi|^2  |w|^2\right ) 
+ \beta_i |\xi|^{2k -2i-2}  \Re  (\overline{ w}\sp w_t) + \gamma_i  |\xi|^{2k -2i-4} |w_t|^2 \right ]
\\
&\quad =\, \left [\left ( \frac{\alpha_i}{q^2} \right )'- \frac {\beta_i} {q^2}\right ]
|\xi|^{2k -2i} |w|^2 + \left [ \alpha_i' + \beta_i\right ]|\xi|^{2k -2i-2} |w_t|^2
\\
& \quad \quad \quad \quad + 
\left [\beta_i' - 2 \frac{\gamma_i}{q^2}\right ]  |\xi|^{2k-2i} \Re (\overline{w}\sp w_t)
 + \gamma_i' |\xi|^{2k-2i- 4} |w_t|^2,
\end{aligned}
\end{equation}
for $ 1\le i \le k-3$. 
Therefore, if $k \ge 4$, we also impose that
\begin{equation}\label {S2}
 {\begin{cases}
\displaystyle{\left ( \frac{\alpha_i}{q^2} \right )'- \frac {\beta_i} {q^2}=0}
 \\
 \displaystyle{\alpha_i' + \beta_i = - \gamma_{i-1}' \phantom{\int^8_8} }
 \\
 \displaystyle{\beta_i' - 2 \frac{\gamma_i}{q^2} =0}
 \end{cases}} \quad \text{for} \quad 1\le i \le k-3,
 \end{equation}
 with $\sp \gamma_0\sp$ given by the system \eqref{S1}.
  
  If the conditions of  \eqref{S1} and \eqref{S2} are satisfied,  the sum of the derivatives of 
  the terms of \eqref{Q} with $\sp 0\le i \le k-3\sp$  simply
 amounts to  $\,  \gamma_{k-3}' |\xi|^{2} |w_t|^2\,$.

\hskip 1mm 

$\quad 3)$  Finally, we consider the terms of \eqref{Q} with  $ \sp i = k-2 \sp $.  In this case we find:
 \begin{equation}  \label{der3}
  \begin{aligned}
\frac d{dt} &\left [\alpha_{k-2} |\xi|^{2} \left ( |w_t|^2  + \frac 1{q^2} |\xi|^2  |w|^2\right ) 
+ \beta_{k-2} |\xi|^{2}  \Re  (\overline{ w}\sp w_t) \right ]
\\
&=\left [\left ( \frac{\alpha_{k-2}}{q^2} \right )'- \frac {\beta_{k-2}} {q^2}\right ]
|\xi|^{4} |w|^2 + \left [ \alpha_{k-2}' + \beta_{k-2}\right ]|\xi|^{2} |w_t|^2 + 
\beta_{k-2}'   |\xi|^{2} \Re (\overline{w}\sp w_t).
\end{aligned}
\end{equation}
Hence, we require 
\begin{equation}\label {S3}
 {\begin{cases}
\displaystyle{\left ( \frac{\alpha_{k-2}}{q^2} \right )'- \frac {\beta_{k-2}} {q^2}=0}
 \\
 \displaystyle{\alpha_{k-2}' + \beta_{k-2} = - \gamma_{k-3}' \,.\phantom{\int^8_8} }
  \end{cases}}
 \end{equation}

  Summarizing up, if the conditions of \eqref{S1}, \eqref{S2} and \eqref{S3} are satisfied, then  the derivative of  $\spp  {\pazocal E}_k(\xi, t)\spp$ is simply
 $\spp \sp \beta_{k-2}'   |\xi|^{2} \Re (\overline{w}\sp w_t)\spp$ as stated in \eqref{derek0}.

 Let us now remark that if we add the equation
 $$ 
 \beta_{k-2}'- 2 \frac{\gamma_{k-2}}{q^2}=0
 $$
 to \eqref{S3}, and set 
 \begin{equation} \label{g-1}
 \gamma_{-1}\equiv 0 \sp , 
 \end{equation} 
 we can summarize
  conditions \eqref{S1}, \eqref{S2} and \eqref{S3} into the following system:
    \begin{equation}\label {S2bis}
 {\begin{cases}
\displaystyle{\left ( \frac{\alpha_i}{q^2} \right )'- \frac {\beta_i} {q^2}=0}
 \\
 \displaystyle{\alpha_i' + \beta_i = - \gamma_{i-1}' \phantom{\int^8_8} }
 \\
 \displaystyle{\beta_i' - 2 \frac{\gamma_i}{q^2} =0}
 \end{cases}} \quad \text{for} \quad 0\le i \le k-2.
 \end{equation}
In this way we have introduced the auxiliary functions $\gamma_{-1}$ and $\gamma_{k-2}$
which are not used in the definition \eqref{Q} of the quadratic form ${\pazocal E}_k$, but we 
can now work with a single system instead of three separate ones. Moreover, it is  clear from \eqref{S2bis} that only the length of the sequence
 $$
  \alpha_0, \beta_0, \gamma_0\spp\sp, \dots, 
   \alpha_{k-2}, \beta_{k-2}, \gamma_{k-2}
  $$
    depends directly on the integer $\sp k\ge 3 \sp $. In other words, for $\sp 0 \le i \le k-2\sp $, the concrete expression of the elements 
    $\alpha_i$, $\beta_i$ and $ \gamma_i$  depends only on $\sp q\sp$ and $\sp i \sp$.

To determine  the functions $\alpha_i$, $\beta_i$ and $\gamma_i$, we note  that the first two equations of \eqref{S2bis} give $\sp \beta_i = - \gamma_{i-1}' - \alpha_i'\sp $ and hence
\begin{equation*}
\left ( \frac{\alpha_i}{q^2} \right )'- \frac {\beta_i} {q^2} = 
\frac 2 q  \left (\frac {\alpha_i} q \right )' + \frac{\gamma_{i-1}'} {q^2} 
\quad \text{for} \quad  0\leq i\leq k-2\sp ,
\end{equation*}
so that system \eqref{S2bis} can be written as
\begin{equation}\label {Siter}
\gamma_{-1} \equiv 0, \quad\quad  
{\begin{cases}
\displaystyle{\left ( \frac{\alpha_i}{q} \right )' = - \frac {\gamma'_{i-1}} {2\sp \sp  q}}
 \\
 \displaystyle{ \beta_i = - \gamma_{i-1}' - \alpha_i'\phantom{\int^8_8} }
 \\
 \displaystyle{\gamma_i=  \frac 1 2 \sp\sp q^2\sp \beta_i'}
 \end{cases}} \quad \text{for} \quad 0\le i \le k-2 \sp .
 \end{equation}
It is therefore easy to see that  system \eqref{S2bis} is  solvable if we assume, together  with \eqref{qno},
\begin{equation} \label{regular}
 q \in C^{2k-2} ([0,T)) \sp .
\end{equation}
Under this assumption,  we also have  $\sp \alpha_i \in C^{2k-2i-2} ([0,T))\sp$, $\sp \beta_i \in C^{2k-2i-3} ([0,T))\sp$ and 
$\sp \gamma_i \in C^{2k-2i-4} ([0,T))\sp$, for $\sp 0 \le i \le k-2\sp$.  
In particular, if \eqref{regular} holds,
$\spi \alpha_0, \beta_0, \gamma_0 \spi $, 
   $\dots $, $\spi \alpha_{k-3}, \beta_{k-3}, \gamma_{k-3}\spi $ and   $\spi \alpha_{k-2}, \beta_{k-2} \spi$ are 
   $\sp C^1\sp$ functions in $\sp [0,T) \sp $, as required in Definition~\ref{Qdef}. 
   
   Furthermore, integrating \eqref{Siter}
we find
\begin{equation}\label {Siter4} 
\gamma_{-1}\equiv 0, \quad \quad\begin{cases}
\displaystyle{
\alpha_i= -  q  \int  \frac{\gamma_{i-1}'}{ 2\sp q }\sp dt}
 \\
 \displaystyle{ \beta_i = -  \left (\gamma_{i-1}-   q \int  \frac{\gamma_{i-1}'} {2\sp q} \sp dt   \right )'
 \phantom{\int^8_8} }
 \\
 \displaystyle{\gamma_i=  -\frac 1 2 \sp\sp q^2\sp 
  \left (\gamma_{i-1}-   q  \int  \frac{\gamma_{i-1}'} {2\sp q} \sp dt   \right )''}
 \end{cases} \quad \text{for} \quad 0\le i \le k-2,
 \end{equation}
where in each of the three expressions of \eqref{Siter4}, for $\sp 0 \le i \le k-2\sp$, we must take the same primitive of  the function $\spp  \frac {\gamma_{i-1}'} {2\sp q}\sp $. 

Finally, from \eqref{Siter4}, it is immediately clear that  the functions 
$\alpha_0, \beta_0$ and $\gamma_0$ depend linearly on an arbitrary   constant of integration 
 $\sp C_0\in \R\sp $;
$\alpha_1, \beta_1, \gamma_1$ depend linearly on $C_0$ and on an additional arbitrary constant of integration $\sp C_1\in \R\sp$; in
general, for $ \sp 1 \le i \le k-2\sp $, the functions  $\alpha_i, \beta_i, \gamma_i$  depend linearly 
  on the previous 
 $\sp i\sp $ arbitrary constants of integration $\spp C_0\spp , \dots, C_{i-1} \in \R\spp $ and  on an additional arbitrary constant of integration $\spp C_i \in \R$.

\medskip
In conclusion, we have the following:
\begin{Prop} \label{prop1}
Let $\sp k \in \N\sp $, $\sp k\ge 3 \sp $. Besides, let us suppose 
\begin{equation} \label{condq}
q=q(t)\in C^{2k-2}([0,T)), \quad q(t) \ne 0 \quad  \text{in}  \quad [0,T).
\end{equation} 
Then, solving the linear systems \eqref{S2bis}, for $\sp 0 \le i \le k-2\sp$,  we obtain functions 
$\alpha_0,\sp  \beta_0 ,  \sp  \gamma_0  $, $\!\dots $, $\alpha_{k-2},\sp  \beta_{k-2},\sp  \gamma_{k-2}$ depending linearly on $\sp k-1 $ arbitrary constants of integration $\sp C_0 \spp , \dots , C_{k-2} \in \R$ and such that
\begin{equation}\label{rego}
\begin{cases}
\alpha_i(t)= \alpha_i (C_0 \spp , \dots , C_i, t) \in C^{2k -2i-2}([0,T))
\\
\displaystyle{\beta_i(t)= \beta_i (C_0 \spp , \dots , C_i, t) \in C^{2k -2i-3}([0,T))  \phantom{\int^8_8}   }
\\
\gamma_i (t)= \gamma_i(C_0 \spp , \dots , C_i, t)\in C^{2k -2i-4}([0,T))
\end{cases}
\end{equation} 
for $\spp 0 \le i \le k-2\sp$. Furthermore, taking into account \eqref{der1}, \eqref{der2} and \eqref{der3}, we have the equality
\begin{equation} \label{derek}
{d\over dt} {\pazocal E}_k(\xi, t)=  \beta_{k-2}' |\xi|^2\Re (\overline{w} w_t) \quad \text{for} \quad t\in [0,T).
\end{equation} 
\end{Prop}


\vskip 0.3cm
\section {The polynomial structure of  $\spp\alpha_i, \beta_i, \gamma_i$.}  \label{polyst}

Under the assumptions of Proposition~\ref{prop1},  in this section we will prove that 
the functions $\sp \alpha_i, \beta_i, \gamma_i \sp $  are 
 polynomials in $\sp q \sp$ and its derivatives $\sp q^{(h) }$ of order
 $\sp h\le 2i \sp $, $\sp h\le 2i + 1\sp $ and $\sp h\le 2i +2  \sp $, respectively. To simplify the explanation, 
below we will set
\begin{equation} \label{int0}
\int \frac {\gamma_{-1}'} {2\sp q} \sp dt=-\sp C_0 \sp ,
\end{equation}
with $\sp C_0 \in \R\sp $, so that from \eqref{Siter4} with $\sp i =0\sp $  we find
\begin{equation} \label{c0}
\alpha_0 = C_0 \sp q, \quad \beta_0 = - C_0 \sp q' , \quad \gamma_0 = - \frac  {C_0} 2 \sp q^2 \sp q'',
\end{equation}
and  we will also require
\begin{equation} \label{non0}
C_0 \ne 0\sp ,
\end{equation} 
so that  $\alpha_0, \beta_0$ and $\gamma_0$ are not identically zero.

Assuming  $\sp k \in \N$, $\sp k \ge 3\sp$ and $\sp q \sp $ as in \eqref{condq}, we use $ \gamma_{-1} \equiv 0, \gamma_0 \spp , \dots, \gamma_{k-3}\sp$ to introduce the following functions:
\begin{Def} \label{defGG}
Let $\sp k \in \N$, $\sp k \ge 3\sp$. We define:
 \begin{equation} \label{defG} 
 G_i \sp : =  \sp \gamma_{i-1} -  q  \int \frac {\gamma_{i-1}' } {2\sp q}\sp \sp dt \sp ,
 \qquad 0\leq i\leq k-2 \sp .
  \end{equation}
  \end{Def}
 For \eqref{int0} and \eqref{non0}, we then have 
 \begin{equation}
  \label{G00}
 G_0 = C_0 \sp q  \quad \text{with} \quad C_0 \ne 0\sp , 
 \end{equation}
and we can write \eqref{Siter4} as
\begin{equation}\label {Siter4bis} 
\begin{cases}
\displaystyle{
\alpha_i= G_i - \gamma_{i-1} }
 \\
 \displaystyle{ \beta_i = -\sp  G_i'
 \phantom{\int^8_8} }
 \\
 {\gamma_i=  - \sp \frac 1 2 \sp\sp q^2\sp 
  G_i''}
 \end{cases} \quad \text{for} \quad 0\le i \le k-2.
 \end{equation}
 \begin{Rem} \label{Greg} 
  Having assumed $\sp q\in  C^{2k-2}([0,T))\sp $, from \eqref{G00} it is clear that $\sp G_0 \in C^{2k-2}([0,T))\sp $. Furthermore, by Proposition~\ref{prop1}, we also know that $\sp \gamma_{i-1} \sp$ is $\sp \in C^{2k- 2i-2}([0,T)) \sp $, 
 for $\sp 1\le i \le k-2$. Therefore, from Definition~\ref{defGG}, we deduce that $\sp G_i \in C^{2k-2i-2}([0,T))\sp $, for 
 $ \sp 0 \le i \le k-2\sp$.
 \end{Rem}

 \vskip 0.1cm
 We now need some technical lemmas to
 prove that every  $\sp G_i \sp $, for  $\sp 0\le i  \le k-2\sp$,  is  a polynomial in $\sp q\sp$ 
 and its derivatives $\sp q^{(h)}\sp $ of order $\sp h \le 2i\sp$. To begin with, we show that  $\spp G_1 \spp , \dots , G_{k-2}\spp$ can be obtained recursively from $\sp G_0\sp$:

\begin{Lemma} \label{recu0} Let $\sp k \in \N$, $\sp k \ge 3\sp$. We have
\begin{equation} \label{rec}
G_i=  - \frac q 4  \int q\sp \sp  G_{i-1}''' \sp \sp dt \, , \qquad 1\le i \le k-2\sp.
\end{equation}
In particular,
\begin{equation}
  \label{Gi/qprimo}
\left (\frac {G_{i}} q\right ) '  = - \frac 1 4 \sp \sp q \sp \sp  G_{i-1}''' \sp , \qquad 1\le i \le k-2\sp.
 \end{equation}
\end{Lemma}

\begin{proof} To prove \eqref{rec} we take the expression of $\gamma_{i-1}$ from the last equation of \eqref{Siter4bis} and get
\begin{equation*}
\gamma_{i-1}' =- q\sp  q'\sp  G_{i-1}'' - \frac 1 2\sp  q^2 \sp G_{i-1}''' 
= - q \left ( q G_{i-1}''  \right )' + \frac 1 2\sp  q^2 \sp G_{i-1}''' ,
\end{equation*}
hence
\begin{equation*}
- \frac {\gamma_{i-1}'} {2\sp q} 
=  \frac  1 2 \left ( q\sp  G_{i-1}''  \right )' - \frac 1 4\sp  q \sp G_{i-1}''' 
\end{equation*}
and 
\begin{equation*}
-  \int \frac {\gamma_{i-1}'} {2\sp q}  \sp  dt
=  \frac  1 2 \sp q\sp  G_{i-1}''   - \frac 1 4\sp  \int  q \sp G_{i-1}'''  \sp dt.
\end{equation*}
Therefore,
\begin{equation*}
-q   \int \frac {\gamma_{i-1}'} {2\sp q}  \sp  dt
=  \frac  1 2  q^2 G_{i-1}''   - \frac q 4\sp  \int  q \sp G_{i-1}'''  \sp dt
=  -\gamma_{i-1}  - \frac q 4\sp  \int  q \sp G_{i-1}'''  \sp dt
\end{equation*}
which, by \eqref{defG}, immediately gives \eqref{rec}.
Finally, \eqref{Gi/qprimo} clearly follows from \eqref{rec}.
\end{proof}
\vskip 0.2cm

In particular, since integration by parts easily yields 
\begin{equation}\label{elint}
\int q\sp q''' \sp dt =  q\sp q'' - \frac 1 2 \sp q'^2 +C \sp,  
\end{equation}
 from  \eqref{G00} and \eqref{rec} we  have
 \begin{equation}\label{G0}
 G_1 =  - \frac q 4  \int q\sp \sp  G_{0}''' \sp \sp dt  =
  \frac {C_0} 4  \left( \frac 1 2 {q'}^2- q\sp  q'' \right ) q\sp  + C_1\sp\sp  q\sp ,
  \end{equation}
where $C_1 \in \R $ is an additional   arbitrary constant of integration. 
By applying \eqref{G0}, we can then recursively obtain also the following
expression of $\sp G_2 \spp ,\dots, G_{k-2}\sp $:
\begin{Lemma} \label{recu3} 
Let $\sp k \in \N$, $\sp k \ge 4\sp$. 
For  $\spp 2 \le  i \le k-2\sp$,
\begin{equation}\label{tt4}
G_i=  - \frac q 4  \left [ q \sp G_{i-1}'' - q' G_{i-1}' + q'' G_{i-1} +
\frac 4 {C_0}\frac {G_1\sp G_{i-1}} {q^2}  + \frac 1 {C_0}\int G_1 \sp G_{i-2}''' \sp\sp  dt \sp \right ] . \end{equation}
 \end{Lemma}
\begin{proof}
To prove \eqref{tt4}, we note that we can write \eqref{G0} as
\begin{equation}\label{G0bis}
 q \sp q'''= - \frac 4 {C_0} \left (\frac {G_1} q  \right )' 
 \end{equation}
since $  \sp C_0 \ne 0\sp$. 

We then integrate \eqref{rec} by parts and use \eqref{G0bis}:
\begin{equation}
\begin{aligned}\label{ipp3}
G_i &=  - \frac q 4  \left [ q\spp G_{i-1}'' - q' G_{i-1}' + q'' G_{i-1} -
 \int q''' \sp \sp G_{i-1}  \sp \sp dt \right ] 
\\
&=  
- \frac q 4  \left [ q\spp G_{i-1}'' - q' G_{i-1}' + q'' G_{i-1} - \int q q'''  \left (\frac{ G_{i-1}} q 
\right  ) \sp \sp dt \right ]
\\
&=  
- \frac q 4  \left [ q\spp G_{i-1}'' - q' G_{i-1}' + q'' G_{i-1} + \frac 4 {C_0}   \int
 \left (\frac{ G_1 } q \right )' \left (\frac{ G_{i-1}} q 
\right  ) \sp \sp dt \right ]
\\
&=  
- \frac q 4  \left [ q\spp G_{i-1}'' - q' G_{i-1}' + q'' G_{i-1} + \frac 4 {C_0} 
  \left (\frac{ G_1 } q \right ) \left (\frac{ G_{i-1}} q 
\right  )  - \frac 4 {C_0}
 \int
 \left (\frac{ G_1 } q \right ) \left (\frac{ G_{i-1}} q 
\right  ) '\sp \sp dt \sp \right ].
\end{aligned}
\end{equation}

From \eqref{Gi/qprimo}, with $\sp i-1\sp$ instead of $\sp i\sp $, we   finally
get \eqref{tt4}.
\end{proof}

\vskip 0.2cm
\begin{Lemma} \label{recu4}Let $\sp k \in \N$, $\sp k \ge 3\sp$.  Then, for all 
$\spp 0 \le i , \sp j \le k-3\sp$ we have:
\begin{equation} \label{tt6}
\int G_i \sp G_j'''\sp  dt = G_i \sp G_j''- G_i'\sp G_j'+ G_i'' \sp G_j +  4 \int 
\left ( \frac {G_{i+1}} q\right )' \left ( \frac {G_{j}} q\right ) dt \sp .
\end{equation}
If $\sp k \in \N$, $\sp k \ge 4\sp$ and  $\spp 0 \le i \le k-3\sp$, then
\begin{equation} \label{tt7}
\int G_i \sp G_j'''\sp  dt = G_i \sp G_j''- G_i'\sp G_j'+ G_i'' \sp G_j +  
4 \sp\sp \frac{G_{i+1} \sp G_j} {q^2}
+ 
 \int  G_{i+1} \sp  \sp  G_{j-1}'''\sp \sp dt\sp ,
\end{equation}
for all  $\spp 1 \le j \le k-3\sp$.
 \end{Lemma}

\begin{proof} 
Let us prove \eqref{tt6}. Integrating by parts and using \eqref{Gi/qprimo} we have
\begin{equation}
\begin{aligned}\label{ipp8}
\int G_i \sp G_j''' dt &=  G_iG_j''- G'_iG_j'+ G_i'' G_j - \int G_i'''\sp G_j \sp\sp  dt
\\
&=  
G_iG_j''- G'_iG_j'+ G_i'' G_j - \int \big ( q {G_i'''}\big ) \sp  \left (\frac{G_j } q \right )\sp\sp  dt
\\
&=  
 G_iG_j''- G'_iG_j'+ G_i'' G_j + 4 \int \left ( \frac {G_{i+1}} q \right)'\sp \left (\frac{G_j } q\right ) \sp \sp dt
.
\end{aligned}
\end{equation}
This gives \eqref{tt6}. Next, to prove \eqref{tt7} we assume $\sp k \ge 4\sp $ and $\sp 1 \le  j\le k-3\sp $. 
Integrating by part and using \eqref{Gi/qprimo} for $i=j$ we get:
\begin{equation}
\begin{aligned}\label{ipp9}
\int \left ( \frac {G_{i+1}} q \right)'\sp \left (\frac{G_j } q\right ) \sp \sp dt &= 
\frac {G_{i+1} \sp G_j} {q^2}
- \int\left ( \frac {G_{i+1}} q \right)\sp \left (\frac{G_j } q\right )' \sp \sp dt
\\
&= \frac {G_{i+1} \sp G_j} {q^2} + \frac 1  4 
 \int\left ( \frac {G_{i+1}} q \right)\sp  \big ( q G_{j-1}'''\big )\sp \sp dt\sp .
\end{aligned}
\end{equation}
Together with \eqref{ipp8} this gives \eqref{tt7}.
\end{proof}

Finally, we  need the following preliminary result:
\begin{Lemma}
\label{clm1}  
Let $\sp k \in \N$, $\sp k \ge 4\sp$. Given  $\sp j \in \N\sp$, $\sp 1 \le  j \le k-3\spp$, let us
suppose that  $\spp G_i\spp $ is of the form 
\begin{equation}\label{exprebis}
G_i = q\sp \Pcal_i  \!\left (q, q', \dots, q^{(2i)}\sp  \right ) \quad \text{for} \quad 0\le i \le j\sp ,
\end{equation}
with $\sp \Pcal_i\sp$ a suitable polynomial of its arguments. Then
\begin{equation} \label{pas2}
\int G_1 \sp G_{j}''' \sp\sp  dt  
\end{equation}
is a   polynomial in $\sp q\sp$ and its derivatives of order $\sp \le 2j+2 \spp $.
\end{Lemma}

\begin{proof}
 We apply Lemma~\ref{recu4}, distinguishing between the case $\sp j\sp$ odd and the 
case $\sp j\sp $ even.

\vskip 2mm
$\;${\it Case $\sp j = 2\sp l+1\sp$.}\; We may suppose $\sp l \ge 1\sp$, because for $\sp l=0\sp $ 
it is sufficient to integrate by parts $\int G_1G_1'''dt$ (as in formula  \eqref{elint}) and use the expression of
$\sp G_1\sp$ given in \eqref{G0}.

Then, assuming $\sp l \ge 1 \sp$, we apply $\sp l\,$ times formula \eqref{tt7}  of
 Lemma~\ref{recu4}. Taking into account that,
 for $\spp 0 \le i \le  j\spp$,  $\sp  G_i \sp $ is of the form \eqref{exprebis}, we find:
\begin{equation} \label{pas3}
\begin{aligned}
\int G_1 \sp G_{j}'''\sp  dt &= G_1 \sp G_{j}''- G_1'\sp G_{j}'+ G_1'' \sp G_{j} +  
4 \sp\sp \frac{G_{2} \sp G_{j}} {q^2}
+ 
 \int  G_{2} \sp  \sp  G_{{j-1}}'''\sp \sp dt
 \\
 &=  \{ \ast \} + 
 \int  G_{2} \sp  \sp  G_{j-1}'''\sp \sp dt
 \\
&=  \{ \ast \} +  G_2 \sp G_{j-1}''- G_2'\sp G_{j-1}'+ G_2'' \sp G_{j-1} +  
4 \sp\sp \frac{G_3 \sp G_{j-1}} {q^2}
+ 
 \int  G_{3} \sp  \sp  G_{j-2}'''\sp \sp dt
  \\
 &=  \{ \ast \} + 
 \int  G_{3} \sp  \sp  G_{{j-2}}'''\sp \spp dt
  \\
 &\ \,\vdots\\
 &=  
 \{ \ast \} + 
 \int  G_{1+l} \sp  \spp  G_{{j-l}}'''\sp \spp dt\sp ,
 \end{aligned}
\end{equation}
where in each step of \eqref{pas3} the symbol $\sp \{ \ast  \}$ denotes a polynomial expression in $\sp q \sp $ 
and its derivatives of order $\sp \le 2 j+2\sp $. Noting that 
$$
 1+ l = j-l= \frac{j+1} 2\sp , 
$$
the last indefinite integral of \eqref{pas3} is elementary:
\begin{equation} \label{pas3bis}
 \int  G_{1+l} \sp  \sp  G_{{j-l}}'''\sp \sp dt = G_{\frac {j+1} 2} \, G_{\frac {j+1} 2}''\spp - 
 \frac  1 2 \left (  G_{\frac {j+1} 2}' \right )^2 + C\sp ,
\end{equation}
with $\sp C \in \R\sp $ an arbitrary constant. 
By \eqref{exprebis}, the right hand side of \eqref{pas3bis} is a polynomial in $\sp q \sp $ and its 
derivatives of order $\sp \le 2(j+1)/2+2\leq2j+2\spp$, since $j\geq1$.

\vskip 3mm
$\;${\it Case $\sp j = 2\sp l\sp$.}\; 
If $\sp l= 1\sp $, we immediately apply formula \eqref{tt6} of Lemma~\ref{recu4} with $\sp i =1\sp $ and $\sp j =2\sp$:

\begin{equation} \label{tt6bis}
\begin{aligned}
\int G_1 \sp G_2'''\sp  dt &= G_1 \sp G_2''- G_1'\sp G_2'+ G_1'' \sp G_2 +  4 \int 
\left ( \frac {G_{2}} q\right )' \left ( \frac {G_{2}} q\right ) dt
\\
&=  G_1 \sp G_2''- G_1'\sp G_2'+ G_1'' \sp G_2 +  2 \left ( \frac{G_2} q \right )^2 +C \sp ,
\end{aligned}
\end{equation}
where $\sp C \in \R \sp$ is an arbitrary constant. 
By \eqref{exprebis}, with $\sp j = 2\sp $, the last expression in \eqref{tt6bis} is 
a polynomial in $\sp q \sp $ 
and its derivatives of order $\sp \le 6\sp = 2j+2\sp$ for $\sp j =2 \sp$.

Let us now suppose $\sp l\ge 2\sp$. We  apply $(\spp l-1)$-times  formula \eqref{tt7} and argue as in \eqref{pas3}. 
We then obtain: 
\begin{equation} \label{pas4}
\int G_1 \sp G_{j}'''\sp  dt =   \{ \ast \} + 
 \int  G_{l} \sp  \sp  G_{{j+1-l}}'''\sp \sp dt\sp ,
\end{equation}
where  $\sp \{ \ast  \}$ is a polynomial expression in $\sp q \sp $ 
and its derivatives of order $\sp \le 2j+2\sp$.
 
Finally, we apply formula \eqref{tt6}  
to the last integral of \eqref{pas4} and find:
\begin{equation} \label{pas5}
\begin{aligned}
\int G_1 \sp G_{j}'''\sp  dt &= \{ \ast \} +  G_l \sp G_{j+1-l}''- G_l'\sp G_{j+1-l}'+ G_l'' \sp G_{j+1-l}
\\
&\quad \quad +  4 \int 
\left ( \frac {G_{l+1}} q\right )' \left ( \frac {G_{j+1-l}} q\right ) dt
\\
&= \{ \ast \} +  4 \int 
\left ( \frac {G_{l+1}} q\right )' \left ( \frac {G_{j+1-l}} q\right ) dt\sp ,
\end{aligned}
\end{equation}
where  $\sp \{ \ast  \}$  still represents a polynomial expression in $\sp q \sp $ 
and its derivatives of order $\sp \le 2 j+2\sp $.
The last indefinite  integral of \eqref{pas5} is elementary because 
$$
 l+1= j+1-l = \frac {j+2}2
$$ 
and hence
\begin{equation}
\label{pas6}
\int 
\left ( \frac {G_{l+1}} q\right )' \left ( \frac {G_{j+1-l}} q\right ) dt = \frac 1  2
\left (\frac {G_{\frac {j+2} 2}} q\right )^2+ C,
\end{equation}
with $\sp C \in \R\sp $ an arbitrary constant.  By \eqref{exprebis}, the right hand side of \eqref{pas6} is clearly
a polynomial in 
$\sp q\sp$ and its derivatives of order $\sp \le j+2\leq 2j +2 \spp$. This concludes the proof.
\end{proof}

\begin{Rem} Note that also $\int G_1G_0'''\sp dt$ is a polynomial in $q$ and its derivatives of
order $\leq2$. Indeed,
from \eqref{Gi/qprimo} we have
\begin{equation}
\int G_1 G_0'''\sp  dt = \int  \left (\frac {G_1} q \right ) \sp q \sp G_0'''\sp  dt = 
-4 \int  \left (\frac {G_1} q \right ) \sp \left (\frac {G_1} q \right ) '\sp  dt = 
-2   \left (\frac {G_1} q \right )^2 + C\sp ,
\end{equation}
with $\sp C \in \R\sp $ an arbitrary constant.  Therefore, taking into account \eqref{G0},   $\sp \int G_1 G_0'''\sp  dt  \sp $ is a polynomial in 
$\sp q, q'\sp $ and $\sp q''\sp$.  
\end{Rem}

We can now prove the main results of this section. As already observed after \eqref{S2bis} for the functions
$\alpha_i$, $\beta_i $ and $ \gamma_i$, we note that the form of 
the elements of the sequence $\sp G_0 \spp , \dots , G_{k-2}\sp$ does not depend on 
$\sp k \sp$. In other words, given $\sp i \in \N$, the expression of the function $\sp G_i\sp $ is the same for every 
integer $ \sp k \ge 3 \sp $ such that $ \sp k-2 \ge i\sp $.

\begin{Th} 
\label{Poly}
Let $\sp k \in \N$, $\sp k \ge 3\sp$. 
Then, for $\sp 0 \le i \le k-2\sp $, we have
\begin{equation} \label{expre}
G_i=  q\spp \Pcal_i \!  \left (q \spp ,  \dots, \sp q^{(2i)}\right )  
\end{equation}
with $\sp \Pcal_i\sp$ a suitable polynomial of its arguments.
\end{Th}

\begin{proof} From  formulas \eqref{G00} and \eqref{G0}, we already know that $\sp G_0\sp $ and $\sp G_1\sp $ are  of the form \eqref{expre}. This means that the statement is certainly true for $\sp k=3\sp$. 

To prove Theorem~\ref{Poly}  for $ \sp k = 4\sp $ it is therefore sufficient to show that  $G_2$ is of the form
\eqref{expre}. Integrating by parts \eqref{rec} for $i=2$, we find
\begin{equation*}
G_2=  - \frac q 4  \left[ q\sp  G_1'' - q' G_1' + q'' G_1 - \int q''' \sp \sp G_1  \sp \sp dt \right ].
\end{equation*}
But, taking into account \eqref{G0bis}, we can write 
\begin{equation*}
q''' \sp G_1 = (q\sp q''') \sp \left (\frac {G_1} q \right )=
 -\frac 4 {C_0}\sp \left ( \frac {G_1} q \right )' \sp  \left ( \frac {G_1} q\right ),
\end{equation*}
so that
\begin{equation} 
\label{recuf1}
\begin{aligned}
G_2= & - \frac q 4  \left [ q \sp G_1'' - q' G_1' + q'' G_1 +\frac 4 {C_0} \int 
\left ( \frac {G_1} q \right )  \left ( \frac {G_1} q\right ) ' \sp dt \right ]\\
=&- \frac q 4  \left [ q\sp  G_1'' - q' G_1' + q'' G_1 +\frac 2 {C_0} \left( \frac {G_1} q \right )^2 +C \right ],
\end{aligned}
\end{equation}
for some $C\in\R$. Therefore $G_2$ is of the form \eqref{expre}.

Having proved the theorem also for $\sp k=4\sp$, we can conclude the proof by induction on $\sp k \ge 4\sp$. Let us then assume that the
statement is true for a given $\sp k_o \ge 4\sp $. This means that
\begin{eqnarray} 
\label{hypot}
G_0 \spp ,  \dots \sp , \spp G_{k_o-2} \quad \text{satisfy \eqref{expre}.}
\end{eqnarray}
Then,  taking $\, k= k_o+1 \sp $, we have to prove that 
 $$
 G_{k_o-1} = q \spp \sp \Pcal_{k_o-1} \!\! \sp \sp \left ( q, \dots , q^{(2k_o -2)} \right ),
 $$
 with $\sp \Pcal_{k_o-1}\sp$ a suitable polynomial of its arguments.

Now, applying Lemma~\ref{recu3}, we can write $\sp G_{k_o-1}\sp $ as 
\begin{equation}\label{tt4bis}
G_{k_o-1}=  - \frac q 4  \left [ q \sp G_{k_o-2}'' - q' G_{k_o-2}' + q'' G_{k_o-2} +
\frac 4 {C_0}\frac {G_1\spp G_{k_o-2}} {q^2}  + \frac 1 {C_0}\int G_1 \spp G_{k_o-3}''' \sp\sp  dt  \right ].
\end{equation}
The expression in the right hand side of \eqref{tt4bis} is  explicitly a multiple of $\sp q \sp $. 
Furthermore, by the induction hypothesis \eqref{hypot}, 
\begin{equation}\label{pas1}
q \sp G_{k_o-2}'' - q' G_{k_o-2}' + q'' G_{k_o-2} +
\frac 4 {C_0}\frac {G_1\spp G_{k_o-2}} {q^2} 
 \end{equation}
is a  polynomial in $\sp q\sp$ and its derivatives of order $\sp \le 2(k_o-2)+2 = 2 k_o-2\sp $. It remains therefore to prove that also the indefinite integral 
\begin{equation} \label{integral}
\int G_1 \spp G_{k_o-3}'''  \sp dt\sp  
\end{equation}
is a polynomial in $\sp q\sp $ and its derivatives of order $\sp \le 2k_o -2\sp$.

For this purpose  we apply Lemma~\ref{clm1} with $\sp k = k_o +1\sp$ and $\sp j = k_o-3\sp$. 
In fact,  since $k_o \ge 4\sp$, we have $\sp 1\le  j= k_o-3  \le k-3\sp$  and by the 
induction hypothesis \eqref{hypot} we know that 
$$
G_0\spp , \dots , G_{k_o-3} \quad \text{satisfy condition \eqref{exprebis}. }
$$
By Lemma~\ref{clm1} it follows that the indefinite integral \eqref{integral} 
is a polynomial in $\sp q\sp $ and its derivatives of order $\sp \le 2(k_o -3) + 2 = 2 k_o-4\sp$.
This concludes the proof.
\end{proof}

\vskip 2mm
From Theorem~\ref{Poly} we immediately obtain:

\begin{Cor} \label{coro} Let  $\sp k\in\N, k \ge 3\sp$. 
Let  $\sp \alpha_i, \beta_i, \gamma_i \sp $  be the functions defined in \eqref{Siter4} 
for $\sp 0 \le i \le k-2\sp $.
 Then,  $\sp \alpha_i, \beta_i, \gamma_i \sp $ are 
 polynomials in $\sp q \sp$ and its derivatives $\sp q^{(h) }$ of order
 $\sp h\le 2i \sp $, $\sp h\le 2i + 1\sp $ and $\sp h\le 2i +2  \sp $, respectively.
\end{Cor}

\begin{proof} 
It follows from \eqref{Siter4bis} and Theorem~\ref{Poly}.
\end{proof}


\vskip 0.3cm
\section{The quadratic form for the Kirchhoff-Pokhozhaev equation} 


We now apply the results of the previous sections to  equation \eqref{KP}. 

\vskip 1mm

$\quad$Let $\sp k\in \N\sp$, $\sp k \ge 3\sp$, and let $\sp u \sp$  be a solution of \eqref{KP} in $\sp \R^n\times [0,T) $ such  that
\begin{equation*}
u \in C^j([0,T); H^{k-j}(\R^n))  \quad  \quad (j= 0,1) \sp , 
\end{equation*}
and \eqref{reg2} holds.  From Remark~\ref{lorder} it follows that
$$
\, t \,\mapsto \, \int_{\R^n} |\nabla u(x,t)|^2 dx \,
$$ 
is a $\sp C^{2k-2}\sp $ function in $[0,T)$.

Taking the partial Fourier transform of $\sp u= u(x,t) \sp$ with respect to the  $x\sp$-variable, i.e.,
\begin{equation*}
w(\xi,t) := (2\pi)^{-n/2} \int_{\R^n} u(x,t) e^{-i x \cdot \xi} \spp dx,  \quad \xi \in \R^n,
\end{equation*}
and noting that 
$$
\,\int_{\R^n} |\nabla u(x,t)|^2 dx= \int_{\R^n} |\xi|^2 |w(\xi,t)|^2 d\xi \spp ,
$$ 
by Plancharel theorem, we are  lead to 
consider the equation
 \begin{equation}
 \label{KP3}
w_{tt} + \frac {|\xi|^2 \spp w} {\big ( a\int_{\R^n} |\xi |^2 |w |^2 d\xi+ b \big )^2 } 
\spp  = 0 ,\quad 
   t\in [0,T).
   \end{equation}
Next, we set
 $q= q(t) = a\|\nabla u \|^2 +b$
as in \eqref{reg2} and  consider the Liouville-type equation 
$$
w_{tt} + \frac 1 {q^2} \sp  |\xi |^2 \sp  w =0\sp ,
$$
for which we introduce the quadratic form ${\pazocal E}_k(\xi, t)$ as in the Definition~\ref{Qdef}. 

In the definition of  $\spp {\pazocal E}_k(\xi, t)\spp $ we use the coefficients
$\alpha_i, \beta_i $ and $\gamma_i  $ obtained  by solving the  system \eqref{S2bis} or, equivalently,  
we take $\alpha_i, \beta_i $ and $\gamma_i$  as in  \eqref{Siter4}, for $\sp 0 \le i \le k-2\sp $. 
Furthermore, we assume that the first constant of integration is such that \eqref{c0} holds with $\sp C_0 \ne 0\sp$. 
Thus, all the results of Section~\ref{polyst} are valid.

Then we introduce the functional 
 \begin{equation}
 \label{newEk}
 E_k= E_k(t) := \int_{\R^n} {\pazocal E}_k(\xi, t)   d\xi \quad \text{for} \quad t \in [0,T).
 \end{equation}
Noting that 
\begin{equation} \label{s1}
\int_{\R^n} |\xi|^2\Re (\overline{w} w_t) d\xi = \frac {1}{2\sp a }\spi q'\sp ,
\end{equation}
from formula \eqref{derek} we easily get
\begin{equation}\label{derE}
\frac{d}{dt} E_k =  \beta_{k-2}' \int_{\R^n} |\xi|^2\Re (\overline{w} w_t) d\xi = 
\frac {1}{2\sp a }\spi \beta_{k-2}'\spi q' \sp .
\end{equation}

Applying Lemma~\ref{clm1} and Theorem~\ref{Poly}, we have the following:

\begin{Lemma}\label{primit} Let $\sp k \in \N$, $\sp k \ge 3\sp$. Then  
\begin{equation}\label{prim}
\beta_{k-2}'\sp  q' = \big ( \beta_{k-2}\sp  q' \big )' + \frac{ d }{dt} \sp Q_k \!
\left (q, \sp q', \dots, \sp q^{(2k-4)} \right )
\end{equation}
where $\spp Q_k\spp $ is a suitable polynomial of its arguments.
\end{Lemma}

\begin{proof} Let us first recall that
$$
\beta_{k-2}= - \spp G_{k-2}'
$$
by \eqref{Siter4bis}.
Then from  \eqref{G0bis} we have:
\begin{equation} \label{start}
\begin{aligned}
\beta_{k-2}'\sp  q' &= \big (\beta_{k-2} \sp q'\big )'- \beta_{k-2}\spp  q''
\\
&= \big (\beta_{k-2}\sp q'\big )'+ \spp G_{k-2}'\spp q''
\\
&= \big (\beta_{k-2}\sp q'\big )'+\spp \big ( G_{k-2} \sp q'' \big )' -  \spp G_{k-2}\spp q''' 
\\
&= \big (\beta_{k-2} \sp q'\big )'+ \spp \big ( G_{k-2}\sp q'' \big )' -
   \left ( \frac { G_{k-2} } q  \right ) \big (q \sp q''' \big )
 \\
&= \big (\beta_{k-2}\sp  q'\big )'+ \spp \big ( G_{k-2}\sp q'' \big )' +
 \frac 4{  C_0}    \left (\frac {G_1} q  \right )' \left ( \frac { G_{k-2} } q  \right ) .
\end{aligned}
\end{equation}

We now first consider the cases $k=3$, where it's enough to apply Theorem~\ref{Poly},
and then face the general case $k\geq 4$, where we also need Lemma~\ref{clm1}.

Let $\sp k=3\sp$ and write \eqref{start} as
\begin{equation} \label{prim3}
\beta_{1}' \sp q' = \big (\beta_{1} q'\big )'+  \frac d {dt}  \left [  \sp  G_{1}\spp q''  +
 \frac 2{   C_0}   {\left (\frac {G_1} {q}  \right )^2}\sp  \right ] .
\end{equation}
By Theorem~\ref{Poly}, we know that $\sp G_1 \sp $ is of the form 
\begin{equation}\label{G11}
G_1= q \spp \Pcal_1 \left (q,q',q'' \sp \right )\sp ,
\end{equation}
with $\sp \Pcal_1\sp$ a suitable polynomial of its arguments. Thus, from the expression \eqref{prim3}, it is clear that formula \eqref{prim} holds 
when $\spp k=3\sp$.

If $\spp k \ge 4\spp$, we continue from the last expression of \eqref{start} by writing
\begin{equation}\label{gener}
\begin{aligned}
\beta_{k-2}' \sp q' &= \big (\beta_{k-2}\spp q'\big )'+  \spp \big ( G_{k-2}\sp  q'' \big )' +
 \frac 4{   C_0} \left ( \frac {G_1 \sp G_{k-2}} {q^2}  \right )' 
 -  \frac 4{   C_0}    \left (\frac {G_1} q  \right ) \left ( \frac { G_{k-2} } q  \right )'
 \\ 
  &= \big (\beta_{k-2}\spp q'\big )'+  \big ( G_{k-2}\sp  q'' \big )' +
 \frac 4{   C_0} \left ( \frac {G_1 \sp G_{k-2}} {q^2}  \right )' 
 +  \frac 1  { C_0}   \spp   {G_1} \spp  G_{k-3}''' \, ,
 \end{aligned}
\end{equation}
because of \eqref{Gi/qprimo} for $i=k-2$.

 By Theorem~\ref{Poly}, it follows that the term
$$
 \big ( G_{k-2}\sp  q'' \big ) +
 \frac 4{  C_0} \left ( \frac {G_1 \sp G_{k-2}} {q^2}  \right ) 
 $$
 is a polynomial in $ \sp q, q', \dots, q^{(2k-4)}\sp $. So, it only remains to show that
 $\sp {G_1} \sp  G_{k-3}''' \sp $ is the derivative of a polynomial in $ \sp q, q', \dots, q^{(2k-4)}\sp $.
This follows from the fact that 
\begin{equation*}
\int G_1 \sp G_{k-3}''' \spp  dt 
\end{equation*}
is  a polynomial in $\sp q\sp$ and its derivatives of order $\sp \le 2k-4 \sp $. To prove this, it is sufficient to 
apply Lemma~\ref{clm1} with $\sp j = k-3\sp$. In fact, $\sp k-3 \ge 1\sp$, because we are assuming $\sp k \ge 4\sp$, and 
 condition  \eqref{exprebis} is  clearly guaranteed by Theorem~\ref{Poly}.
\end{proof}


\vskip 0.2cm
\section{Proof of Theorem~\ref{T1}} \label{concl}

\vskip 3mm
After having proved Lemma~\ref{primit}, it is now easy to prove Theorem~\ref{T1}. 
To begin with, we set $\sp C_0= 1\sp $ in \eqref{c0}. That is,  we choose  
\begin{equation}
\label{newalpha0}
\alpha_0 = q \sp, \quad  \; \beta_0 = -q' \sp , \quad \; \gamma_0 = - \frac 1 2 \sp q^2 \sp q'' \sp .
\end{equation}
From \eqref{derE} and Lemma~\ref{primit} we see that  the functional
\begin{equation*} 
I_k= E_k - \frac 1 {2\sp a} \beta_{k-2}\sp  q'  - \frac 1 {2\sp a } \sp Q_k \!
\left (q, \sp q', \dots, \sp q^{(2k-4)} \right )
\end{equation*}
is invariant by time evolution in $\sp [0,T)\sp $.

Taking into account \eqref{newEk} and \eqref{Q}, and noting that 
$$
\beta_{k-2} \int_{\R^n} |\xi|^2\Re (\overline{w} w_t) d\xi \,  =  \,
\frac {1}{2\sp a }\spi \beta_{k-2} \spi q' 
$$
we can express 
$\sp I_k \sp $ as
\begin{equation}
 \begin{aligned}\label{IK}
I_k \, :=  \;  &\alpha_0 \int_{\R^n} |\xi|^{2k-2} \left(|w_t|^2  + \frac 1{q^2}|\xi|^2  |w|^2\right )  \sp d\xi
+  \beta_0  \int_{\R^n}|\xi|^{2k-2}  \Re  (\overline{ w} \sp w_t) \sp  d\xi 
\\
 &
 +\sum_{i=1}^{k-2}
 \alpha_i \int_{\R^n} |\xi|^{2k-2i-2} \left(|w_t|^2  + \frac 1{q^2}|\xi|^2  |w|^2\right )  \sp d\xi
 \\
 &
+\sum_{i=1}^{k-3}  \beta_i  \int_{\R^n}|\xi|^{2k-2i-2}  \Re  (\overline{ w} \sp w_t) \sp  d\xi 
 + \sum_{i=0}^{k-3} \gamma_i  \int_{\R^n} |\xi|^{2k -2i-4} |w_t|^2 \sp d \xi
\\
& 
- \frac 1 {2\sp a } \spp Q_k \!
\left (q, \sp q', \dots, \sp q^{(2k-4)} \right ),
 \end{aligned}
\end{equation}
where the second summation of \eqref{IK}, that is,
 $ \, \sum_{i=1}^{k-3}  \beta_i  \int_{\R^n}|\xi|^{2k-2i-2}  \Re  (\overline{ w} \sp w_t) \sp  d\xi 
 \, $, is empty if $\sp k = 3\sp $. By Plancharel theorem, taking into account \eqref{newalpha0},
 we can also write:
 
 \begin{equation}
 \begin{aligned}\label{IK2}
I_k \, :=  \;  &  q {\|\nabla^{k-1} u_t\|^2} + \frac{\|\nabla^{k} u\|^2} q -
 {q'} \!\int_{\R^n} \nabla^{k-1} u  \cdot \!\nabla^{k-1} u_t\spp\sp dx 
\\
 &
 +\sum_{i=1}^{k-2}
 \alpha_i\left(    {\|\nabla^{k-i-1} u_t\|^2} + \frac{\|\nabla^{k-i} u\|^2} {q^2}  \right ) 
+\sum_{i=1}^{k-3}  \beta_i  \int_{\R^n} \nabla^{k-i-1} u  \cdot \!\nabla^{k-i-1} u_t\spp\sp dx 
\\ 
& + \sum_{i=0}^{k-3} \gamma_i \spp  \| \nabla^{k-i-2} u_t\|^2 
- \frac 1 {2\sp a } \spp Q_k \!
\left (q, \sp q', \dots, \sp q^{(2k-4)} \right ),
 \end{aligned}
\end{equation}
where 
\begin{equation}
 \begin{aligned}\label{IK3}
J_k \, :=  \; 
 &
 \sum_{i=1}^{k-2}
 \alpha_i\left(    {\|\nabla^{k-i-1} u_t\|^2} + \frac{\|\nabla^{k-i} u\|^2}{ q^2}  \right ) 
+\sum_{i=1}^{k-3}  \beta_i  \int_{\R^n} \nabla^{k-i-1} u  \cdot \!\nabla^{k-i-1} u_t\spp\sp dx 
\\ 
& + \sum_{i=0}^{k-3} \gamma_i \spp  \| \nabla^{k-i-2} u_t\|^2 
- \frac 1 {2\sp a } \spp Q_k \!
\left (q, \sp q', \dots, \sp q^{(2k-4)} \right )
 \end{aligned}
\end{equation}
 is a functional of order $\sp \le k-1\sp$, in the sense of Definition~\ref{lf}.

 In fact, by Corollary~\ref{coro}, we know that $\sp \alpha_i\sp $, $ \sp \beta_i\sp $ and $ \sp \gamma_i\sp $ are polynomials 
 in $\sp q \sp $ and its derivatives $\sp q^{(h)}\sp $ of order $\sp h\le 2i\sp$, $\sp h\le 2i+1\sp$  and 
 $\sp h\le 2i+2 \sp$, 
 respectively. Therefore, the maximum order of the derivatives of $\sp q\sp$ in the  expression \eqref{IK3} is 
 $ \spp \le 2k-4\sp $ and,  taking into account Remark~\ref{lorder}, we  deduce that $\sp J_k\sp$ is a functional of order 
 $\sp \le k-1\sp$.
 
 This concludes the proof of Theorem~\ref{T1}.
 
 
\vspace{4mm}
{\bf Acknowledgments.}
The first author is member of the Gruppo Nazionale per l'Analisi Ma\-te\-ma\-ti\-ca, la
Probabilit\`a e le loro Applicazioni (GNAMPA) of the Istituto Nazionale di Alta Ma\-te\-ma\-ti\-ca (INdAM) and was partially supported by the Italian Ministry of University and Research, under PRIN2022 (Scorrimento) ``Anomalies in partial differential equations and applications", code: 2022HCLAZ8\_002, CUP: J53C24002560006.


\end{document}